\documentclass[11pt, reqno]{amsart}

\usepackage{amssymb,latexsym,amsmath,amsfonts,amsthm}
\usepackage{mathrsfs}
\usepackage{enumitem}
\usepackage[usenames]{color}
\usepackage{hyperref}
\usepackage{comment}
\usepackage{tikz}

\allowdisplaybreaks

\numberwithin{equation}{section}

\theoremstyle{definition}
\newtheorem{definition}{Definition}[section]
\newtheorem{question}[definition]{Question}

\theoremstyle{remark}
\newtheorem{remark}[definition]{Remark}

\theoremstyle{plain}
\newtheorem{theorem}[definition]{Theorem}

\newtheorem{lemma}[definition]{Lemma}
\newtheorem{proposition}[definition]{Proposition}

\newcommand{\zt}{\zeta}

\newcommand{\bas}{\boldsymbol{\epsilon}}

\newcommand{\bdy}{\partial}
\newcommand{\OM}{\Omega}
\newcommand{\D}{\mathbb{D}}

\newcommand{\B}{\mathbb{B}}

\newcommand{\bcdot}{\boldsymbol{\cdot}}

\newcommand\squee[1]{{}^{\raisebox{-1pt}{$\scriptstyle{{#1}}$}}\!s}
\newcommand{\spd}{\widehat{s}}

\newcommand{\Cn}{\mathbb{C}^n}

\newcommand{\C}{\mathbb{C}}

\newcommand{\wtW}{\mathcal{W}}

\newcommand{\re}{{\sf Re}}

\begin{document}

\title[Universal bounds for squeezing functions]{Explicit universal bounds for squeezing functions \\
of ($\mathbb{C}$-)convex domains}

\author{Gautam Bharali}
\address{Department of Mathematics, Indian Institute of Science, Bangalore 560012, India}
\email{bharali@iisc.ac.in}

\author{Nikolai Nikolov}
\address{Institute of Mathematics and Informatics, Bulgarian Academy of Sciences, Acad. G. Bonchev 8,
1113 Sofia, Bulgaria
\vspace{1mm}
\newline Faculty of Information Sciences, State University of Library Studies and Information
Technologies, Shipchenski prohod 69A, 1574 Sofia, Bulgaria}
\email{nik@math.bas.bg}

\begin{abstract}
We prove two separate lower bounds\,---\,one for nondegenerate convex domains and the other for
nondegenerate $\C$-convex (but not necessarily convex) domains\,---\,for the squeezing function that hold
true for all domains in $\C^n$, for a fixed $n\geq 2$, of the stated class. We provide explicit
expressions in terms of $n$ for these estimates.
\end{abstract}


\maketitle

\vspace{-4.5mm}
\section{Introduction and statement of results}\label{S:intro}
The \emph{squeezing function} of a domain $D\subset \C^n$, denoted by $s_D$, is defined as
\[
  s_D(z)\,:=\,\sup\{s_D(z;F) \mid F: D\to \B^n \; \text{is an injective holomorphic map with $F(z)=0$}\}
\]
(in this paper, $\B^n$ will denote the open Euclidean unit ball in $\Cn$ with centre $0$) where, for
each $F: D\to \B^n$ as above, $s_D(z;F)$ is given by
\[
  s_D(z;F)\,:=\,\sup\{r>0 : r\B^n\subset F(D)\}.
\]
(Note that if $D$ is unbounded, then the above-mentioned class of biholomorphic maps may be empty, in
which case $s_D(z) := 0$.) By definition, $s_D$ is invariant under biholomorphic maps: i.e., for
any biholomorphic map $\Phi$ defined on $D$, $s_{\Phi(D)}\circ \Phi = s_D$. For this reason, the
squeezing function has proven to be of considerable utility. This function was introduced
by Deng--Guan--Zhang in \cite{dengGuanZhang:spsfbd12}. It is closely related to notions
introduced by Liu--Sun--Yau \cite{liuSunYau:cmmsRsI04} and Yeung \cite{yeung:gdusf09}, most notably that
of holomorphic homogeneous regularity. A domain $D$ is said to be \emph{holomorphic homogeneous regular}
if $\inf_{z\in D}s_D(z) > 0$. A number of facts concerning the complex geometry of a domain
$D$ are known if $D$ is holomorphic homogeneous regular: we refer the reader to Section~1 of
\cite{kimZhang:uspbccdCn16} and to the references therein.
\smallskip

Over the years, several families of domains have been shown to be holomorphic homogeneous regular.
We will not list these families here. Instead, we shall focus on two classes of domains that are known
to be holomorphic homogeneous regular: namely, the class of nondegenerate convex domains and the class
of nondegenerate $\C$-convex domains (we say that a domain $D$ is \emph{nondegenerate} if $D$
contains no complex lines; clearly, $s_D\equiv 0$ if $D$ contains a complex line).
\smallskip

The goal of this paper is to examine more closely an interesting property of the
above-mentioned classes as a whole. Specifically, these two classes admit \emph{universal lower bounds}
for the squeezing function: by this, we mean that there exists a constant $\kappa_n>0$ such that
\[
  s_D(z) \geq \kappa_n \quad \forall z\in D \; \; \text{and for \textbf{every} $D\subset \Cn$},
\]
where
\begin{itemize}[leftmargin=24pt]
  \item $D$ is a nondegenerate convex domain, or
  \item $D$ is a nondegenerate $\C$-convex domain.
\end{itemize}
The existence of universal lower bounds follows, in the convex case, from the work of
Frankel \cite{frankel:psmp_affine-geom1991} (the function $s_D$ was introduced more recently but the
results in \cite{frankel:psmp_affine-geom1991} can be reinterpreted in terms of $s_D$). In the
$\C$-convex case, it first was hinted at, to the best of our knowledge, in \cite{nikolovAndreev:bbsfCdpd17}
by Nikolov--Andreev. In this paper, we revisit this phenomenon to find explicit estimates
for $\kappa_n$.To this end, an observation by Nikolov--Andreev proves to be more useful.
We make precise the quantitative basis for this observation. In the process, we find two sets\,---\,for
convex and $\C$-convex domains, respectively\,---\,of \textbf{explicit} universal lower bounds. We ought
to say more about this, but (after one explanatory remark) we first present the two theorems
just alluded to.
\smallskip

Determining when a domain $D\subset \Cn$, $n\geq 2$, is holomorphic homogeneous regular requires
estimating $s_D$. To this end, it is sometimes useful to work with the squeezing function
$\squee{\mathcal{D}}_D$\,---\,i.e., the \emph{squeezing function based on $\mathcal{D}$} of the
domain $D$\,---\,where $\mathcal{D}$ is a unit ball, with respect to some norm $\|\bcdot\|_{\mathcal{D}}$,
such that $\mathcal{D}$ is homogeneous. See, for instance, \cite{guptaPant:sfcp20, rongYang:Figsf22}, in
which these auxiliary squeezing functions are introduced. With this in mind, we also present explicit
universal estimates for $\squee{\D^n}_D$. Now, $\squee{\D^n}_D$ is defined in a manner similar to $s_D$,
with $\D^n$ replacing $\B^n$ in the definition above. (Here, $\D$ denotes the open unit disc in $\C$ with
centre $0$.) For simplicity of notation, we shall write $\spd_D :=\,\squee{\D^n}_D$ and set
$c_n := \sqrt{4^n\,-\,1}/\sqrt{3}$. With these words, we present:

\begin{theorem}\label{T:sf_cvx}
For any nondegenerate convex domain $D\subset \Cn$, $n\geq 2$, the
squeezing functions $s_D$ and $\spd_D$ have the following universal lower bounds:
\[
   s_D(z)\,\geq\,\frac{1}{\sqrt{n}\,(2c_n+1)} \; \; \;\text{and}
  \; \; \; \spd_D(z)\,\geq\,\frac{1}{2^{n+1}-1} \quad\forall z\in D.
\]
\end{theorem}

\begin{theorem}\label{T:sf_C-cvx}
For any nondegenerate $\C$-convex domain $D\subset \Cn$, $n\geq 2$, the
squeezing functions $s_D$ and $\spd_D$ have the following universal lower bounds:
\begin{align*}
  s_D(z)\,&\geq\,\frac{1}{\sqrt{n}(\sqrt{c_n}+\sqrt{c_n+1})^2}\,>\,\frac{1}{\sqrt{n}\,(4c_n+2)}
   \; \; \;\text{and} \\
  \spd_D(z)\,&\geq\,\frac{1}{(\sqrt{2^n}+\sqrt{2^n-1})^2}\,>\,\frac{1}{2^{n+2}-2}
  \quad\forall z\in D.
\end{align*}
\end{theorem}

\begin{remark}
Recall that a domain $D\subset \Cn$ is said to be \emph{$\C$-convex} if any non-empty intersection of $D$ with a
complex line $\Lambda\subseteq \Cn$ is a simply connected domain in $\Lambda$. Clearly, any convex domain
is $\C$-convex, but the converse is false. The point of Theorem~\ref{T:sf_cvx} is that by specifically
considering convex domains, the lower bounds obtained for them are stronger than those obtained
in Theorem~\ref{T:sf_C-cvx}. The second (and weaker) set of lower bounds given in Theorem~\ref{T:sf_C-cvx}
are meant to convey, in some sense, how much stronger the lower bounds in Theorem~\ref{T:sf_cvx} are: namely,
the lower bounds in Theorem~\ref{T:sf_cvx} are, due to convexity, close to the double of those
in Theorem~\ref{T:sf_C-cvx}. This is a bit reminiscent of the way the estimate provided by K{\"o}be's
theorem for convex univalent maps on $\D$ is the double of the estimate for univalent maps in general
defined on $\D$.
\end{remark}

\begin{remark}
The methods used in proving the above theorems assume the fact that the domains of
interest are domains in $\Cn$, $n\geq 2$. In fact, the estimates above are
uninformative when taking $n=1$. Moreover, we already know that for any planar domain
$D$ of the type considered in the theorems above, $s_D\equiv 1$ (by the Riemann Mapping
Theorem).
\end{remark}

We must also remark that the inequalities in Theorems~\ref{T:sf_cvx} and~\ref{T:sf_C-cvx} are strict.
This can be deduced by carefully examining the proofs. Since the question of a universal
lower bound being attained as an equality is of potential interest, we shall elaborate upon the
latter statements after proving Theorems~\ref{T:sf_cvx} and~\ref{T:sf_C-cvx}; see Section~\ref{S:sf_C-cvx}.
\smallskip

The inspiration for the above theorems is the proof of \cite[Theorem~1]{nikolovAndreev:bbsfCdpd17}.
The latter theorem mentions the existence of universal lower bounds for $s_D$ for $\C$-complex domains,
but the key factor for this is somewhat implicit in its proof. This ``key factor'' is a set of simple
inequalities, which are presented in Proposition~\ref{P:npz} below. It is these simple inequalities that
make possible:
\begin{itemize}[leftmargin=24pt]
  \item the explicit estimates in the theorems above, 
  \item the improved estimates, even while staying within the framework of
  \cite[Theorem~1]{nikolovAndreev:bbsfCdpd17}, for convex domains in Theorem~\ref{T:sf_cvx}.
\end{itemize}
We ought to mention that \cite[Theorem~1.1]{frankel:psmp_affine-geom1991} by Frankel can be
interpreted as a statement about the existence of universal lower bounds for the squeezing
functions of nondegenerate convex domains. Its proof relies on an argument involving affine
scalings. Interestingly, the scaling approach proves to be somewhat of a deterrent to obtaining
explicit universal lower bounds for the squeezing function, although it serves several other
purposes in \cite{frankel:psmp_affine-geom1991} (also see \cite{kimZhang:uspbccdCn16})! In fact,
Theorem~\ref{T:sf_cvx} can also be seen as a non-scaling-based proof of the holomorphic homogeneous
regularity of a nongenerate convex domain in $\Cn$, $n\geq 2$.
\smallskip

Proposition~\ref{P:npz} will play a crucial role in the proof of Theorems~\ref{T:sf_cvx}
and~\ref{T:sf_C-cvx}. The next section is devoted to this result and to a useful lemma
that follows from Proposition~\ref{P:npz}. The proofs of Theorems~\ref{T:sf_cvx}
and~\ref{T:sf_C-cvx} are presented in Sections~\ref{S:sf_cvx} and~\ref{S:sf_C-cvx},
respectively.
\medskip

\section{Some supporting results}\label{S:complex}
This section is devoted to a few results needed in the proofs of Theorems~\ref{T:sf_cvx}
and~\ref{T:sf_C-cvx}. We first introduce a result that is one of the key ingredients of the
latter proofs. It paraphrases a useful calculation that is given in
\cite{nikolovPflugZwonek:eimCcd11} by Nikolov \emph{et al.} (also see
\cite{nikolovPflug:eBkmcdCn03} for some aspects of this calculation).
But first, we must record that a domain $D\subset \Cn$ is said to be
\emph{linearly convex} if for each point $a\in (\Cn\setminus D)$, there exists a complex hyperplane
through $a$ that does not intersect $D$. It turns out (see \cite[Chapter~2]{a-p-s:cccf04}) that
any $\C$-convex domain is linearly convex, which is a fact that we will need.
We will also need some notation, in presenting which we shall follow the notation used
in \cite[Section~1]{nikolovPflugZwonek:eimCcd11}. Let $D\subset\Cn$
be a nondegenerate domain and let $z_0\in D$. There exist $\C$-linear subspaces
$H_0,\dots, H_{n-1}$, and points $a^1,\dots, a^n\in \bdy{D}$, such that
\begin{align*}
  H_0\,&:=\,\Cn, \\
  \|a^1\!-\!z_0\|\,&=\,\sup\left\{r>0 : B_r(H_0, z_0)\subset D\right\}, \\
  H_j\,&:=\,\Cn \ominus {\rm span}_{\C}\{(a^1-z_0),\dots, (a^j-z_0)\}, &&\text{\hspace{-1.7cm}and} \\
  \|a^{j+1}\!-\!z_0\|\,&=\,\sup\left\{r>0 : B_r(H_j, z_0)\subset D\right\}, &&\text{\hspace{-1.7cm}$j=1,\dots,n-1$},
\end{align*}
where, given a complex subspace $V\subset \Cn$, $\Cn\ominus V$ denotes the orthogonal complement
of $V$ in $\Cn$ with respect to the standard Hermitian inner product on $\Cn$, and
\[
  B_r(V, z_0)\,:=\,\{z\in \Cn : (z-z_0)\in V \text{ and } \|z-z_0\|< r\}.
\]

Without loss of generality, we can assume that $z_0 = 0$.
Now, suppose $D$ is convex. For each $a^j$, fix a real hyperplane $\wtW_{j-1}$
such that $(a^j + \wtW_{j-1})$ is a supporting hyperplane of $D$ at $a^j$. If $D$
is $\C$-convex (but not necessarily convex), it is linearly convex. Then, for each $a^j$, there exists a complex
hyperplane $W_{j-1}$ such that $(a^j + W_{j-1})$ does not intersect $D$. If $D$ is convex,
it is also $\C$-convex, and we note that for each supporting hyperplane $(a^j + \wtW_{j-1})$,
\begin{equation}\label{E:real_to_cplx}
  W_j\,:=\,\wtW_j\cap i\wtW_j, \quad j=0,\dots, n-1.
\end{equation}
is a complex hyperplane with the property stated in the previous sentence.

\begin{proposition}\label{P:npz}
Let $D$ be a nondegenerate $\C$-convex domain in $\Cn$, $n\geq 2$,
and assume $0\in D$. Let $a^j\in \bdy{D}$, $j=1,\dots, n,$ be the points described above (taking $z_0 = 0$).
Let ${\sf T}$ be the invertible linear transformation given by
\[
  {\sf T}(z)\,:=\,\sum_{j=1}^n\frac{\langle z, a^j\rangle}{\|a^j\|^2}\,\bas_j.
\]
Fix complex hyperplanes $W_j$, $j=0,\dots, n-1,$ as described above (which, when $D$
is convex, are determined by real hyperplanes $\wtW_j$ as given by \eqref{E:real_to_cplx}).
Then, there exists a $\C$-linear transformation
$A$ such that $[\,A\,]_{{\rm std.}} := [\alpha_{j,\,k}]$ is a lower triangular matrix each of whose diagonal
entries is $1$ and such that:
\begin{enumerate}[leftmargin=30pt, label=$(\alph*)$]
  \item If $D$ is convex, then
  \[
    A\circ {\sf T}(a^j+\wtW_{j-1})\,=\,\{(Z_1,\dots, Z_n)\in \Cn : \re(Z_j)=1\},
    \quad j=1,\dots, n.
  \]
  \item If $D$ is $\C$-convex (but not necessarily convex), then
  \[
    A\circ {\sf T}(a^j+ W_{j-1})\,=\,\{(Z_1,\dots, Z_n)\in \Cn : Z_j=1\},
    \quad j=1,\dots, n.
  \]
\end{enumerate}
In both cases, $|\alpha_{j,1}|,\dots , |\alpha_{j,\,j-1}|\leq 1$ for $j=2,\dots, n$.
\end{proposition}

\begin{remark}
To clarify the notation used: each $\bas_j$ is a vector in $(\bas_1,\dots, \bas_n)$\,---\,the standard
ordered basis of $\Cn$; given a linear transformation $T: \Cn\to \Cn$, $[\,T\,]_{{\rm std.}}$ denotes
the matrix representation of $T$ relative to the standard basis; and $\langle\bcdot,\bcdot\rangle$
denotes the standard Hermitian inner product.
\end{remark}
\begin{proof}
The existence of the matrix $A$ is precisely the construction given in
\cite[Section~1]{nikolovPflugZwonek:eimCcd11}. So, all that needs to be proved are the inequalities
for $|\alpha_{j,\,k}|$. These inequalities are claimed in \cite{nikolovAndreev:bbsfCdpd17}.
Since these are crucial for the explicit estimates made below, we provide
a proof of these inequalities. Notice that, by construction, the discs $\D\bas_j\subset {\sf T}(D)$,
$j = 1,\dots, n$. Then, from \cite[Lemma~15]{nikolovPflugZwonek:eimCcd11}, it follows (and which is
obvious when $D$ is convex) that
\[
  \Delta^n\,:=\,\{(w_1,\dots, w_n)\in \Cn: |w_1|+\dots+|w_n|<1\} \subset {\sf T}(D).
\]
Suppose there exists some $j: 2\leq j\leq n$ and $k: 1\leq k\leq j-1$ such that
$|\alpha_{j,\,k}| > 1$. Then $|1/\alpha_{j,\,k}| < 1$, whence the point
$p := (1/\alpha_{j,\,k})\bas_k \in \Delta^n$. On the other hand, if we write
\[
  (Z_1,\dots, Z_n)\,:=\,A(p),
\]
then $Z_j=1$. Whether $D$ is convex or not, this means that
\[
  A(p)\in A\circ {\sf T}(a^j+W_{j-1})\subset \Cn\setminus A\circ {\sf T}(D),
\]
which implies that $p\notin {\sf T}(D)$. But this contradicts the fact that $p\in \Delta^n$.
This establishes the desired inequalities.
\end{proof}

From the last proposition, we can deduce the following estimates. Recall that
$c_n = \sqrt{4^n-1}/\sqrt{3}$.

\begin{lemma}\label{L:contained}
Let $D\subset \Cn$ be as in Proposition~\ref{P:npz} and, for a choice of $a^1,\dots, a^n\in
\bdy{D}$ as described above, let $A$ be as given by Proposition~\ref{P:npz}. Then
\begin{align}
  \frac{1}{2^n\!-\!1}\D^n\,&\subset\,A(\Delta^n), \label{E:pd_contained} \\
  (1/c_n)\B^n\,&\subset\,A(\Delta^n). \label{E:ball_contained}
\end{align}
\end{lemma}
\begin{proof}
Let $\mathscr{D}$ denote either $\D^n$ or $\B^n$. For $c>0$,
\begin{equation}\label{E:contained_goals}
  c\mathscr{D}\subset A(\Delta^n)\,\iff\,cA^{-1}(\mathscr{D})\subset \Delta^n.
\end{equation}
Following the notation of Proposition~\ref{P:npz}, let $Z_j$ denote the $j$-th coordinate of a point
in $\Cn$. Set
\[
  (w_1,\dots, w_n)\,:=\,A^{-1}(Z).
\]
A simple argument by induction shows that:
\begin{itemize}[leftmargin=24pt]
  \item[$(*)$] For $j\geq 2$, $w_j$ is a linear form depending on $Z_1,\dots, Z_j$ where the
  coefficient of $Z_j$ is $1$ and the coefficient of each $Z_k$, $k=1,\dots, j-1$, comprises
  $2^{j-k-1}$ monomials that are products of the terms in the set
  $\cup_{\nu=2}^j\{\alpha_{\nu,1},\dots, \alpha_{\nu,\,\nu-1}\}$.
\end{itemize}
From $(*)$ and the last assertion of Proposition~\ref{P:npz}, we have:
\[
  |w_j|\,\leq\,|Z_j|+\sum_{k=1}^{j-1}2^{j-k-1}|Z_k|, \quad j=2,\dots, n.
\]
Therefore
\begin{equation}\label{E:bound}
  |w_1|+\dots +|w_n|\,\leq\,|Z_1|+\sum_{j=2}^n\Big(|Z_j|+
  \sum\nolimits_{k=1}^{j-1}|Z_k|\Big)\,=\,\sum_{j=1}^n 2^{n-j}|Z_j|.
\end{equation}
With these preparations, we can now give:
\medskip

\noindent{\emph{The proof of \eqref{E:pd_contained}:} By \eqref{E:bound}, for $c>0$:
\[
  c\big(\,|Z_n|+2|Z_{n-1}|+\dots+2^{n-1}|Z_1|\,\big)< 1 \; \; \forall Z\in
  \D^n\,\Rightarrow\,cA^{-1}(\D^n)\subset \Delta^n.
\]
From this, and since $|Z_j|<1$ for $j=1,\dots, n$ whenever $Z\in \D^n$, we get:
\[
  cA^{-1}(\D^n)\subset \Delta^n \text{ whenever } c(2^n - 1)\leq 1.
\]
This, in view of \eqref{E:contained_goals}, implies \eqref{E:pd_contained}.}
\medskip

\noindent{\emph{The proof of \eqref{E:ball_contained}:} Since
\[
  |Z_n|+2|Z_{n-1}|+\dots+2^{n-1}|Z_1|\,\leq\,\sqrt{4^{n-1}+ 4^{n-2}\dots+ 1}\,\|Z\|,
\]
\eqref{E:bound} implies that for $c>0$:
\[
  c_nc\|Z\| < 1 \; \; \forall Z\in \B^n\,\Rightarrow\,cA^{-1}(\B^n)\subset \Delta^n.
\]
From this, and since $\|Z\|<1$ whenever $Z\in \B^n$, we get:
\[
  cA^{-1}(\B^n)\subset \Delta^n \text{ whenever } c_nc\leq 1.
\]
This, in view of \eqref{E:contained_goals}, implies \eqref{E:ball_contained}.}
\end{proof}

\section{The proof of Theorem~\ref{T:sf_cvx}}\label{S:sf_cvx}
To prove Theorem~\ref{T:sf_cvx}, we first need a simple lemma.

\begin{lemma}\label{L:new_ball}
Let $\psi(\zt):=\zt/(2-\zt)$ and $\Psi(z)\,:=\,\big(\psi(z_1),\dots, \psi(z_n)\big).$ Let $c\in (0,1)$ and
$\tau(c) := c/(2+c)$. Then
\begin{equation}\label{E:ball_rad}
  \tau(c)\B^n\subset \Psi(c\B^n).
\end{equation}
\end{lemma}
\begin{proof}
Note that
\[
  \Psi(c\B^n)\,=\,\left\{w\in \Cn : \bigg(\left|\frac{2w_1}{1+w_1}\right|^2+\dots+
  		\left|\frac{2w_n}{1+w_n}\right|^2\bigg)^{1/2} < c\right\}.
\]
We also have the simple estimate:
\begin{equation}\label{E:mobius_est}
  \frac{2|\zt|}{|1+\zt|}\,\leq\,\frac{2|\zt|}{1-|\zt|} \quad \forall \zt\in \D.
\end{equation}
From the above, it follows that
\[
  \bigg(\left|\frac{2w_1}{1+w_1}\right|^2+\dots+
  		\left|\frac{2w_n}{1+w_n}\right|^2\bigg)^{1/2}\,<\,\frac{2\|w\|}{1-r}
		\quad \forall w\in r\B^n.
\]
From the above, it follows that for any $r>0$ such that
\[
  \frac{2r}{1-r}\,\leq\,c,\mbox{ that is , } 0<r\,\leq\, c/(2+c),
  \]
one has that $r\B^n\subset \Psi(c\B^n)$.
\end{proof}

We are now in a position to give the

\begin{proof}[The proof of Theorem~\ref{T:sf_cvx}]
Fix $z\in D$. Since $s_{(\bcdot)}$ and $\spd_{(\bcdot)}$ are invariant under biholomorphic maps, we
may assume without loss of generality that $0\in D$ and that $z=0$. Furthermore, it suffices to estimate
$s_{A\circ {\sf T}(D)}(0)$ and $\spd_{A\circ {\sf T}(D)}(0)$, where ${\sf T}$ and $A$ are as
given by Proposition~\ref{P:npz}. By construction, the discs $\D\bas_j\subset {\sf T}(D)$,
$j = 1,\dots, n$. Then, owing to convexity,
\begin{equation}\label{E:diam_cvx_1}
  \Delta^n \subset {\sf T}(D).
\end{equation}
By Proposition~\ref{P:npz}-$(a)$,
\[
  A\circ {\sf T}(D)\,\subset\,\{(Z_1,\dots, Z_n)\in \Cn : \re{Z_j}<1, \ j=1,\dots, n\}.
\]
It is well known that $\Psi(Z) := (\psi(Z_1),\dots, \psi(Z_n))$\,---\,where $\psi$ is as
introduced in Lemma~\ref{L:new_ball}\,---\,defines a biholomorphic map
from $\{(Z_1,\dots, Z_n)\in \Cn : \re{Z_j}<1, \ j=1,\dots, n\}$ onto $\D^n$.
\smallskip

Let us first estimate $s_{D}$. By the last statement, it suffices to estimate
$s_{\Psi\circ A\circ {\sf T}(D)}(0)$. By \eqref{E:diam_cvx_1}, \eqref{E:ball_contained}, and
Lemma~\ref{L:new_ball}
\[
  \tau(1/c_n)\B^n\,\subset\,\Psi\circ A\circ {\sf T}(D).
\]
Finally, since a scaling by $1/\sqrt{n}$ maps $\D^n$ into
$\B^n$, by definition (and since $s_{\Psi\circ A\circ {\sf T}(D)}(0) = s_D(z)$), we conclude
that
\[
  s_D(z)\,\geq\,\frac{1}{\sqrt{n}(2c_n+1)} \quad\forall z\in D.
\]

By an argument analogous to the one in the previous paragraph, but using
\eqref{E:diam_cvx_1}, \eqref{E:pd_contained}, and taking $n=1$ in Lemma~\ref{L:new_ball}, we get
\[
  \tau(1/(2^n-1))\D^n\,=\,\frac{1}
  {2^{n+1}\!-\!1}\D^n\,\subset\,\Psi\circ A\circ {\sf T}(D).
\]
Then, by definition (recall that $\Psi\circ A\circ {\sf T}(D)\subset \D^n$),
\[
  \spd_D(z)\,\geq\,\frac{1}{2^{n+1}\!-\!1} \quad\forall z\in D.
\]
\end{proof}

\section{The proof of Theorem~\ref{T:sf_C-cvx}}\label{S:sf_C-cvx}
To prove our theorem, we need the following lemma.

\begin{lemma}\label{L:new_ball_C}
Let $\OM_1,\dots, \OM_n$ be simply connected planar domains such that $0\in \OM_j$ and $1\in \bdy\OM_j$,
$j=1,\dots, n$. For each $j=1,\dots, n,$ let $\varphi_j$ be a Riemann map $\varphi_j : (\OM_j, 0)\to (\D, 0)$.
Let $c\in(0,1]$ be such that
$c\D \subset \OM_j$ for each $j=1,\dots, n$. Let $\Phi : \prod_{j=1}^n\OM_j\to \Cn$ be defined as
\[
  \Phi(z)\,:=\,\big(\varphi(z_1),\dots, \varphi(z_n)\big).
\]
Set $\rho(c) := c/(1+\sqrt{1+c})^2$. Then,
\[
  \rho(c)\B^n\subset \Phi(c\B^n).
\]
\end{lemma}
\begin{proof}
Write $f_j := \varphi_j^{-1}$. Since $1\in \bdy\OM_j$, ${\rm dist}(0, \bdy\OM_j)\leq 1$
for each $j=1,\dots, n$. Thus, for each $j$, the K{\"o}be $1/4$ Theorem implies that
$|f_j^\prime(0)|\,\leq\,4$. We now apply the K{\"o}be Distortion Theorem along with the estimate for
$|f_j^\prime(0)|$ to get
\begin{equation}\label{E:Kobe_j}
  |f_j(\zt)|\,\leq\,\frac{4|\zt|}{(1-|\zt|)^2} \quad \forall \zt\in \D,
\end{equation}
for each $j=1,\dots, n$.
\smallskip

Now, note that
\[
  \Phi(c\B^n)\,=\,\left\{w\in \Cn : \sqrt{|f_1(w_1)|^2+\dots+|f_n(w_n)|^2}<c\right\}.
\]
By \eqref{E:Kobe_j}, we have
\[
  \sqrt{|f_1(w_1)|^2+\dots+|f_n(w_n)|^2}\,\leq\,\frac{4}{(1-r)^2}\|w\|
  \quad \forall w\in r\B^n.
\]
From the last two inequalities, it follows that for any $r\in(0,1)$ such that
\[
  \frac{4r}{(1-r)^2}\,\leq\,c,\mbox{ that is, }0<r\,\leq\,\rho(c),
  \]
one has that $r\B^n\subset \Phi(c\B^n)$
\end{proof}

We are now in a position to give the

\begin{proof}[The proof of Theorem~\ref{T:sf_C-cvx}]
Fix $z\in D$. Since $s_{(\bcdot)}$ and $\spd_{(\bcdot)}$ are invariant under biholomorphic maps, we
may assume without loss of generality that $0\in D$ and that $z=0$. Furthermore, it suffices to estimate
$s_{A\circ {\sf T}(D)}(0)$ and $\spd_{A\circ {\sf T}(D)}(0)$, where ${\sf T}$ and $A$ are as
given by Proposition~\ref{P:npz}. By construction, the discs $\D\bas_j\subset {\sf T}(D)$,
$j = 1,\dots, n$. As $D$ is $\C$-convex, it is linearly convex. Thus, owing to
\cite[Lemma~15]{nikolovPflugZwonek:eimCcd11},
\begin{equation}\label{E:diam_cvx}
  \Delta^n \subset {\sf T}(D).
\end{equation}
Let $\pi_j$ denote the projection of $\Cn$ onto the $j$-th coordinate.
Because $A\circ {\sf T}(D)$ is $\C$-convex, it follows from \cite[Theorem~2.3.6]{a-p-s:cccf04}
that $\OM_j:=\pi_j(A\circ {\sf T}(D))$ is simply connected. Furthermore, it follows from
Proposition~\ref{P:npz}-$(b)$ that
\[
  1\in \bdy\OM_j, \quad j=1,\dots, n.
\]
Thus, $\OM_1,\dots, \OM_n$ satisfy the conditions stated in Lemma~\ref{L:new_ball_C}.
For each $j$, let $\varphi_j$ be a Riemann map $\varphi_j : (\OM_j, 0)\to (\D, 0)$. Finally,
define $\Phi : \prod_{j=1}^n\OM_j\to \Cn$ as
\[
  \Phi(z)\,:=\,\big(\varphi(z_1),\dots, \varphi(z_n)\big).
\]

Let us first estimate $s_{D}$. As $\Psi$ is a biholomorphic map, it suffices to estimate
$s_{\Psi\circ A\circ {\sf T}(D)}(0)$. By \eqref{E:diam_cvx} and \eqref{E:ball_contained}
\[
  (1/c_n)\B^n\,\subset\,A\circ {\sf T}(D)\,\subset\,\prod_{j=1}^n\OM_j.
\]
Then Lemma~\ref{L:new_ball_C} gives
\[
  \rho(1/c_n)\B^n\,\subset\,\Phi\circ A\circ {\sf T}(D)\,\subset\,\D^n.
\]
Finally, since a scaling by $1/\sqrt{n}$ maps $\D^n$ into
$\B^n$, by definition (and since $s_{\Psi\circ A\circ {\sf T}(D)}(0) = s_D(z)$), we conclude
that
\[
  s_D(z)\,\geq\,\frac{\rho(1/c_n)}{\sqrt{n}}\,=\,\frac{1}{(\sqrt{c_n} + \sqrt{c_n+1})^2}
  \quad\forall z\in D.
\]

By an argument analogous to the one in the previous paragraph, but using
\eqref{E:diam_cvx}, \eqref{E:pd_contained}, and taking $n=1$ in Lemma~\ref{L:new_ball_C}, we get
\[
  \rho(1/(2^n-1))\D^n\,\subset\,\Phi\circ A\circ {\sf T}(D).
\]
Then, by definition (recall that $\Phi\circ A\circ {\sf T}(D)\subset \D^n$),
\[
  \spd_D(z)\,\geq\,\rho(1/(2^n-1)) = \frac{1}{(\sqrt{2^n}+\sqrt{2^n-1})^2} \quad\forall z\in D.
\]

As for the two weaker inequalities: they follow from the strict concavity of $\sqrt{\bcdot}$ on
$[0,+\infty)$.
\end{proof}

We end with a discussion on a point made in passing in Section~\ref{S:intro}:
i.e., that the inequalities in Theorems~\ref{T:sf_cvx} and~\ref{T:sf_C-cvx} are strict. Since the
classes of domains, for which universal lower bounds are given by these theorems, are so
well-studied, the question of attaining equalities is of interest. To make precise this question (we
shall focus on $s_D$; analogous remarks apply to $\spd_D$), some notation: let
$\mathscr{C}_n$ denote either the class of all nondegenerate convex domains or all nondegenerate
$\C$-convex domains in $\Cn$, $n\geq 2$, and let $\kappa_n > 0$ be a universal lower bound for $s_D$
for the class $\mathscr{C}_n$. One may ask: \emph{does there exist a domain $D\in \mathscr{C}_n$ and 
$z_0\in D$ such that $\kappa_n = s_D(z_0)$?}
\smallskip

Another way to state our assertion in Section~\ref{S:intro} is that, taking the \textbf{specific} values
of $\kappa_n$ given by Theorems~\ref{T:sf_cvx} and~\ref{T:sf_C-cvx}, the answer to the above question
(or its analogue for $\spd_D$) is in the negative. To see why this is so, let us focus on $s_D$
\footnote{\,{The reader will notice that once the condition \eqref{E:intrscn} is stated, a quick argument
can be given \textbf{for $\boldsymbol{s_D}$} since $\bdy\B^n$ is a smooth manifold. The argument that we
provide instead works for both $s_D$ and $\spd_D$.}}. A necessary condition for the answer to the above question
to be, ``Yes'' is that, for some domain $D\in \mathscr{C}_n$ (and some $z_0\in D$), in addition to the
relation \eqref{E:ball_contained}, we must also have
\begin{equation}\label{E:intrscn}
  (1/c_n)\overline{\B^n}\cap \bdy(A\circ {\sf T}(D))\,\neq\,\varnothing
\end{equation}
(note that $A$ is determined by $D$ and $z_0$). Given the manner in which the scaling constant $(1/c_n)$ is
arrived at in the proof of Lemma~\ref{L:contained}, a necessary condition for the intersection in
\eqref{E:intrscn} to be non-empty is that, for the $n\times n$ matrix $[\alpha_{j,\,k}]$ described by
Proposition~\ref{P:npz}, one has
\[
  |\alpha_{j,\,k}|\,=\,1 \quad \forall k: 1\leq k\leq j-1 \text{ and } j: 2\leq j\leq n.
\]
But the above is impossible because it is easy to show\,---\,using an argument similar to that in the proof
of Proposition~\ref{P:npz}\,---\,that
\[
 \nexists\,j=2,\dots, n \text{ such that } |\alpha_{j,\,1}| =\dots = |\alpha_{j,\,j-1}| = 1.
\]
Analogous remarks apply to $\spd_D$.
\smallskip

Since the focus of the above discussion is Lemma~\ref{L:contained}, it suggests that there is scope for
better universal lower bounds. That being said, it would require \emph{very different} ideas to replace
the closing arguments in the proofs of our main theorems\,---\,especially of Theorem~\ref{T:sf_C-cvx}.
Thus, a question that may be more tractable than the one stated above is as follows:

\begin{question}
Let $\mathscr{C}_n$, $n\geq 2$, be as introduced above. Compute
\[
  \kappa_n\,:=\,\inf_{D\in \mathscr{C}_n}\,\inf_{z\in D}s_D(z), \quad\text{and}
  \quad \widehat{\kappa}_n\,:=\,\inf_{D\in \mathscr{C}_n}\,\inf_{z\in D}{\spd}_D(z).
\]
Is there a domain $D\in \mathscr{C}_n$ such that $\inf_{z\in D}s_D(z) = \kappa_n$ or 
$\OM\in \mathscr{C}_n$ such that $\inf_{z\in \OM}\spd_{\OM}(z) = \widehat{\kappa}_n$? 
\end{question} 
\vspace{-1mm}

\section*{Acknowledgements}\vspace{-2.5mm}
We are grateful to the anonymous referee of this work for their helpful suggestions on our exposition.
G.~\!Bharali is supported by a DST-FIST grant (grant no.~TPN-700661). N.~\!Nikolov was partially supported
by the Bulgarian National Science Fund, Ministry of Education and Science of Bulgaria, under contract KP-06-Rila/2.

\end{document}